\documentclass[12pt]{article}
\usepackage[latin1]{inputenc}
\usepackage{cmap}
\usepackage{lmodern}

\usepackage{amssymb, amsmath, amsthm}
\usepackage[a4paper,top=25mm,bottom=25mm,left=25mm,right=25mm]{geometry}
\usepackage{etex}

\usepackage{authblk} % for headings
\usepackage{pifont}
\usepackage{graphicx}
\usepackage[usenames,dvipsnames,svgnames,table]{xcolor}
\usepackage[figuresright]{rotating}
\usepackage{xtab} % tackle the long tables
\usepackage{longtable} % tackle the long tables
\usepackage{multirow}
\usepackage{footnote}
\usepackage[stable]{footmisc}
\usepackage{chngpage} % allows for temporary adjustment of side margins
\usepackage{pdflscape} % landscape environment

\usepackage{pgfplots}
\pgfplotsset{compat=1.14}
\usepackage{setspace}

\usepackage{array}
\newcolumntype{K}[1]{>{\centering\arraybackslash$}p{#1}<{$}}

\makesavenoteenv{tabular}
\usepackage{tabularx}
\usepackage{booktabs}
\usepackage{threeparttable}
\usepackage[referable]{threeparttablex} % footnotes in tabu
\newcolumntype{R}{>{\raggedleft\arraybackslash}X}
\newcolumntype{L}{>{\raggedright\arraybackslash}X}
\newcolumntype{C}{>{\centering\arraybackslash}X}
\newcolumntype{A}{>{\columncolor{gray!25}}C}
\newcolumntype{a}{>{\columncolor{gray!25}}c}

\usepackage{dcolumn} % alignment to decimal points
\newcolumntype{.}{D{.}{.}{-1}}

\usepackage{tikz}
\usetikzlibrary{arrows}
\usepackage[semicolon]{natbib} % ,numbers
\usepackage[hyphens]{url}
\usepackage{hyperref} % [hidelinks]
\hypersetup{
  colorlinks   = true,    % Colours links instead of ugly boxes
  urlcolor     = blue,    % Colour for external hyperlinks
  linkcolor    = blue,    % Colour of internal links
  citecolor    = red      % Colour of citations
}
\usepackage{microtype}
\usepackage[justification=centerfirst]{caption}

\usepackage[labelformat=simple]{subcaption}

\DeclareCaptionLabelFormat{parenthesis}{(#2)}
\makeatletter
\renewcommand\p@subfigure{\arabic{figure}.}
\makeatother

\DeclareCaptionLabelFormat{parenthesis}{(#2)}
\makeatletter
\renewcommand\p@subtable{A.\arabic{table}.}
\makeatother

\usepackage{enumitem}

\setlist[itemize]{leftmargin=3\parindent}
\setlist[enumerate]{leftmargin=2\parindent}

\theoremstyle{plain}

\newtheorem{corollary}{Corollary}[section]
\newtheorem{lemma}{Lemma}[section]
\newtheorem{proposition}{Proposition}[section]

\theoremstyle{definition}
\newtheorem{axiom}{Axiom}%[section]

\newtheorem{definition}{Definition}[section]

\theoremstyle{remark}

\def\keywords{\vspace{.5em} % Add keywords
{\noindent \textit{Keywords}:\,}}

\def\JEL{\vspace{.5em} % Add keywords
{\noindent \textbf{\emph{JEL} classification number}:\,}}

\def\AMS{\vspace{.5em} % Add keywords
{\noindent \textbf{\emph{MSC} classes}:\,}}

\author{\href{http://scholar.google.hu/citations?user=xje88NkAAAAJ&}{L\'aszl\'o Csat\'o}\thanks{~e-mail: laszlo.csato@uni-corvinus.hu} }
\affil{Institute for Computer Science and Control, Hungarian Academy of Sciences (MTA SZTAKI) \\
Laboratory on Engineering and Management Intelligence, Research Group of Operations Research and Decision Systems}
\affil{Corvinus University of Budapest (BCE) \\
Department of Operations Research and Actuarial Sciences}
\affil{Budapest, Hungary}
\title{Eigenvector Method and rank reversal in group decision making revisited}
\date{\today}

\begin{document}

\maketitle

\begin{abstract}
It has been shown recently that the Eigenvector Method may lead to strong rank reversal in group decision making, that is, the alternative with the highest priority according to all individual vectors may lose its position when evaluations are derived from the aggregated group comparison matrix.
We give a minimal counterexample and prove that this negative result is a consequence of the difference of the rankings induced by the right and inverse left eigenvectors.

\JEL{C44, D71}

\AMS{90B50, 91B08}

\end{abstract}

\section{Introduction}

\citet{PerezMokotoff2016} have presented an example that the alternative with the highest priority according to all individual vectors is not necessarily the best on the basis of aggregated group comparison matrix if Eigenvector Method is used for deriving priorities.
The current paper gives an axiomatic discussion with a focus on this property called group-coherence for choice. Our main contributions are:
(1) we reveal that the violation of group-coherence for choice by the Eigenvector Method is rooted in right-left asymmetry, that is, the difference of the rankings induced by the right and inverse left eigenvectors;
(2) we provide a minimal example (in the number of alternatives), which is an improvement on the proof of \citet[Thereom~3.2~(b)]{PerezMokotoff2016}.

The axioms introduced below are strongly connected to rank reversal, a well-known phenomenon that has been extensively studied in the AHP literature \citep{SaatyVargas1984b, BarzilaiGolany1994, Schenkerman1994, WangElhag2006, WangLuo2009, MalekiZahir2013, Hou2016, Koczkodajetal2016}.

%Similarly, we think a significant shortcoming of any weighting method can be the existence of two alternatives $i$ and $j$ such that all individuals (weakly) prefer $i$ to $j$, but $i$ is not (weakly) preferred to $j$ on the basis of aggregated pairwise comparison matrix.
%The current paper gives an axiomatic discussion of this property and proves that its violation by the Eigenvector Method is rooted in left-right asymmetry, that is, the difference of left and right eigenvectors.

\section{Ranking from pairwise comparison matrices} \label{Sec2}

Assume that $n$ alternatives should be evaluated with respect to a given criterion.
Their pairwise comparisons are known such that $a_{ij}$ is an assessment of the relative importance of alternative $i$ with respect to alternative $j$.

Let $\mathbb{R}^{n}_+$ and $\mathbb{R}^{n \times n}_+$ denote the set of positive (with all elements greater than zero) vectors of size $n$ and matrices of size $n \times n$, respectively.

\begin{definition} \label{Def21}
\emph{Pairwise comparison matrix}:
Matrix $\mathbf{A} = \left[ a_{ij} \right] \in \mathbb{R}^{n \times n}_+$ is a \emph{pairwise comparison matrix} if $a_{ji} = 1/a_{ij}$ for all $1 \leq i,j \leq n$.
\end{definition}

The set of all pairwise comparison matrices of size $n \times n$ is denoted by $\mathcal{A}^{n \times n}$.
%A pairwise comparison matrix is well-defined by its elements above the diagonal.

Let $\mathbf{1} \in \mathcal{A}^{n \times n}$ be the pairwise comparison matrix with all elements equal to $1$.
% of size $n \times n$

The set $\mathcal{A}^{n \times n}$ is closed under some modifications of pairwise comparison matrices.

\begin{definition} \label{Def22}
\emph{Transformation of row multiplication}:
Let $\mathbf{A} \in \mathcal{A}^{n \times n}$ be a pairwise comparison matrix and $1 \leq i \leq n$ be an alternative. A \emph{transformation of row multiplication} on $i$ by $\alpha$ provides the pairwise comparison matrix $\hat{\mathbf{A}} \in \mathcal{A}^{n \times n}$ such that $\hat{a}_{ij} = \alpha a_{ij}$ and $\hat{a}_{ji} = a_{ji} / \alpha$ for all $j \neq i$, but $\hat{a}_{k \ell} = a_{k \ell}$ if $\{ k, \ell \} \cap i = \emptyset$ or $k = \ell = i$.
\end{definition}

This transformation changes the $i$th row of a pairwise comparison matrix by multiplying all elements outside the diagonal by $\alpha$ (and dividing all column elements outside the diagonal by $\alpha$). In other words, the importance of alternative $i$ compared to all other alternatives is multiplied by $\alpha$.

\begin{definition} \label{Def23}
\emph{Weight vector}:
Vector $\mathbf{w}  = \left[ w_{i} \right] \in \mathbb{R}^n_+$ is a \emph{weight vector} if $\sum_{i=1}^n w_{i} = 1$.
\end{definition}

The set of all weight vectors of size $n$ is denoted by $\mathcal{R}^{n}$.

\begin{definition} \label{Def24}
\emph{Weighting method}:
Function $f: \mathcal{A}^{n \times n} \to \mathcal{R}^{n}$ is a \emph{weighting method}.
\end{definition}

A weighting method associates a weight vector to each pairwise comparison matrix.

One of the most popular weighting method is the following.

\begin{definition} \label{Def25}
\emph{Eigenvector Method} ($EM$) \citep{Saaty1980}:
The weight vector $\mathbf{w}^{EM} (\mathbf{A}) \in \mathcal{R}^n$ for any pairwise comparison matrix $\mathbf{A} \in \mathcal{A}^{n \times n}$ is given by
\[
\mathbf{A} \mathbf{w}^{EM}(\mathbf{A}) = \lambda_{\max}(\mathbf{A}) \mathbf{w}^{EM}(\mathbf{A}),
\]
where $\lambda_{\max}(\mathbf{A})$ denotes the maximal eigenvalue, also known as Perron eigenvalue, of $\mathbf{A}$.
\end{definition}

A number of articles have discussed axioms for weighting methods \citep{Fichtner1984, Fichtner1986, BarzilaiCookGolany1987, CookKress1988, Bryson1995, Barzilai1997, Dijkstra2013, Csato2017b}. However, according to our knowledge, the following property has been not introduced yet.

\begin{axiom} \label{Axiom1}
\emph{Invariance to row multiplication} ($IRM$):
Let $\mathbf{A}, \hat{\mathbf{A}} \in \mathcal{A}^{n \times n}$ be two pairwise comparison matrices such that $\hat{\mathbf{A}}$ is obtained from $\mathbf{A}$ through a transformation of row multiplication on $i$ by $\alpha$.
Weighting method $f: \mathcal{A}^{n \times n} \to \mathcal{R}^{n}$ is \emph{invariant to row multiplication} if $f_i \left( \hat{\mathbf{A}} \right) / f_j \left( \hat{\mathbf{A}} \right) =  \alpha f_i \left( \mathbf{A} \right) / f_j \left( \mathbf{A} \right)$ for all $j \neq i$.
%and $f_j \left( \hat{\mathbf{A}} \right) / f_k \left( \hat{\mathbf{A}} \right) =  f_j \left( \mathbf{A} \right) / f_k \left( \mathbf{A} \right)$ for all $j, k \neq i$.
\end{axiom}

$IRM$ requires the weight of alternative $i$ to follow the change of its relative importance compared to all other alternatives.

\begin{lemma} \label{Lemma21}
The Eigenvector Method satisfies invariance to row multiplication.
\end{lemma}

\begin{proof}
Let $\mathbf{w} \in \mathbb{R}^n_+$ be the vector given by $w_i = \alpha w^{EM}_i \left( \mathbf{A} \right)$ and $w_j = w^{EM}_j \left( \mathbf{A} \right)$ for all $j \neq i$.
It is shown that $\mathbf{w}$ is a (positive) right eigenvector of $\hat{\mathbf{A}}$, so it should be associated with the maximal eigenvalue.

Consider the equation for the $i$th row:
\[
\sum_{k=1}^n \hat{a}_{ik} w_k = \hat{a}_{ii} w_i + \sum_{k \neq i} \hat{a}_{ik} w_k = \alpha \sum_{k=1}^n a_{ik} w^{EM}_k \left( \mathbf{A} \right) = \alpha \lambda_{\max}(\mathbf{A}) w^{EM}_i \left( \mathbf{A} \right) = \lambda_{\max}(\mathbf{A}) w_i.
\]

Consider the equation for the $j (\neq i)$th row:
\[
\sum_{k=1}^n \hat{a}_{jk} w_k = \hat{a}_{ji} w_i +  \sum_{k \neq i}^n \hat{a}_{jk} w_k = \sum_{k=1}^n a_{jk} w^{EM}_k \left( \mathbf{A} \right) = \lambda_{\max}(\mathbf{A}) w^{EM}_j \left( \mathbf{A} \right) = \lambda_{\max}(\mathbf{A}) w_j.
\]

It  implies that $w^{EM}_i \left( \hat{\mathbf{A}} \right) / w^{EM}_j \left( \hat{\mathbf{A}} \right) = w_i / w_j = \alpha w^{EM}_i \left( \mathbf{A} \right) / w^{EM}_j \left( \mathbf{A} \right)$ and the maximal eigenvalues of $\mathbf{A}$ and $\hat{\mathbf{A}}$ coincide. 
\end{proof}

Weighting methods are often used only to derive a \emph{ranking} of the alternatives. Ranking $\succeq$ is a weak order, a complete ($i \succeq j$ or $i \preceq j$ for all $1 \leq i,j \leq n$), reflexive ($i \succeq i$ for all $1 \leq i \leq n$) and transitive (for all $1 \leq i,j,k \leq n$: if $i \succeq j$ and $j \succeq k$, then $i \succeq k$) binary relation on the set of alternatives.
The set of all rankings with respect to $n$ alternatives is denoted by $\mathfrak{R}^{n}$.

\begin{definition} \label{Def26}
\emph{Ranking method}:
Function $g: \mathcal{A}^{n \times n} \to \mathfrak{R}^{n}$ is a \emph{ranking method}.
\end{definition}

A ranking method associates a ranking of the alternatives to each pairwise comparison matrix.
All weighting method induces a ranking method.
For instance, the ranking $\succeq^{EM} \in \mathfrak{R}^{n}$ according to the Eigenvector Method is defined by $i \succeq^{EM}_\mathbf{A} j$ if $w_i^{EM}(\mathbf{A}) \geq w_j^{EM}(\mathbf{A})$ for any pairwise comparison matrix $\mathbf{A} \in \mathcal{A}^{n \times n}$, that is, if the weight of alternative $i$ is at least as high as the weight of alternative $j$ by the Eigenvector Method.

%Some articles have implicitly used weighting methods as a ranking method, for example, \citet{CsatoRonyai2016} have introduced a condition on the ranking of alternatives from an incomplete pairwise comparison matrix.

\section{Properties of ranking methods} \label{Sec3}

The following axioms are specified for ranking methods, that is, they only deal with the relative importance of alternatives. According to our knowledge, the literature on pairwise comparison matrices has been dealt with similar properties only implicitly such as the linear order preservation \citep{CsatoRonyai2016}, or group-coherence for choice \citep{PerezMokotoff2016}. On the other hand, they have been extensively studied in social choice (see, for example, \citet{ChebotarevShamis1998a} and \citet{Gonzalez-DiazHendrickxLohmann2013}).

\begin{axiom} \label{Axiom2}
\emph{Anonymity} ($ANO$):
Let $\mathbf{A} = \left[ a_{ij} \right] \in \mathcal{A}^{n \times n}$ be a pairwise comparison matrix, $\sigma: \{ 1,2, \dots ,n \} \rightarrow \{ 1,2, \dots ,n \}$ be a permutation on the set of alternatives, and $\sigma(\mathbf{A}) = \left[ \sigma(a)_{ij} \right] \in \mathcal{A}^{n \times n}$ be the pairwise comparison matrix obtained from $\mathbf{A}$ by this permutation such that $\sigma(a)_{ij} = a_{\sigma(i) \sigma(j)}$.
Ranking method $g: \mathcal{A}^{n \times n} \to \mathfrak{R}^n$ is \emph{anonymous} if $i \succeq^g_\mathbf{A} j \iff \sigma(i) \succeq^g_{\sigma(\mathbf{A})} \sigma(j)$ for all $1 \leq i,j \leq n$.
\end{axiom}

$ANO$ makes the ranking of alternatives independent of their labels. This property was called 'comparison order invariance' by \citet{Fichtner1984} in the case of weighting methods when the weight of alternative $\sigma(i)$ in $\sigma(\mathbf{A})$ is required to be equal to the weight of alternative $i$ in $\mathbf{A}$.

\begin{definition} \label{Def31}
\emph{Aggregation of pairwise comparison matrices}:
Let $\mathbf{A}^{(1)} = \left[ a_{ij}^{(1)} \right] \in \mathcal{A}^{n \times n}$, $\mathbf{A}^{(2)} = \left[ a_{ij}^{(2)} \right] \in \mathcal{A}^{n \times n}$, $\dots$, $\mathbf{A}^{(k)} = \left[ a_{ij}^{(k)} \right] \in \mathcal{A}^{n \times n}$ be any pairwise comparison matrices. Their \emph{aggregate} is the pairwise comparison matrix $\mathbf{A}^{(1)} \oplus \mathbf{A}^{(2)} \oplus \dots \oplus \mathbf{A}^{(k)} = \left[ \sqrt[k]{a_{ij}^{(1)} a_{ij}^{(2)} \cdots a_{ij}^{(k)}} \right] \in \mathcal{A}^{n \times n}$.
\end{definition}

In other words, aggregation is equivalent to taking the geometric mean of all corresponding matrix elements.
This transformation preserves the reciprocity condition. \citet{AczelSaaty1983} use an axiomatic approach in order to prove that geometric mean is the only reasonable aggregation procedure, a conclusion widely accepted in the decision-making community.

\begin{axiom} \label{Axiom3}
\emph{Aggregation invariance} ($AI$):
Let $\mathbf{A}^{(1)},\mathbf{A}^{(2)}, \dots , \mathbf{A}^{(k)} \in \mathcal{A}^{n \times n}$ be any pairwise comparison matrices. Let $g: \mathcal{A}^{n \times n} \to \mathfrak{R}^n$ be a ranking method such that $i \succeq^g_{\mathbf{A}^{(\ell)}} j$ for all $1 \leq \ell \leq k$.
$g$ is called \emph{aggregation invariant} if $i \succeq^g_{\mathbf{A}^{(1)} \oplus \mathbf{A}^{(2)} \oplus \dots \oplus \mathbf{A}^{(k)}} j$, furthermore, $i \succ^g_{\mathbf{A}^{(1)} \oplus \mathbf{A}^{(2)} \oplus \dots \oplus \mathbf{A}^{(k)}} j$ if $i \succ^g_{\mathbf{A}^{(\ell)}} j$ for at least  one $1 \leq \ell \leq k$.
\end{axiom}

$AI$ is a general and intuitive condition of group decision making: if alternative $i$ is not worse (better) than $j$ according to all decision-makers, this relation should be preserved after their preferences are aggregated.

\citet{PerezMokotoff2016} introduce a weaker property called group-coherence for choice where alternative $i$ should have the highest priority in each pairwise comparison matrices.

\begin{axiom} \label{Axiom4}
\emph{Group-coherence for choice} ($GCC$):
Let $\mathbf{A}^{(1)},\mathbf{A}^{(2)}, \dots , \mathbf{A}^{(k)} \in \mathcal{A}^{n \times n}$ be any pairwise comparison matrices. Let $g: \mathcal{A}^{n \times n} \to \mathfrak{R}^n$ be a ranking method such that $i \succeq^g_{\mathbf{A}^{(\ell)}} j$ for all $1 \leq j \leq n$ and $1 \leq \ell \leq k$.
$g$ is called \emph{group-coherence for choice} if $i \succeq^g_{\mathbf{A}^{(1)} \oplus \mathbf{A}^{(2)} \oplus \dots \oplus \mathbf{A}^{(k)}} j$ for all $1 \leq j \leq n$, furthermore, $i \succ^g_{\mathbf{A}^{(1)} \oplus \mathbf{A}^{(2)} \oplus \dots \oplus \mathbf{A}^{(k)}} j$ for all $1 \leq j \leq n$ if $i \succ^g_{\mathbf{A}^{(\ell)}} j$ for all $1 \leq j \leq n$ in the case of at least one $1 \leq \ell \leq k$.
\end{axiom}

\begin{definition} \label{Def32}
\emph{Opposite of a pairwise comparison matrix}:
Let $\mathbf{A} = \left[ a_{ij} \right] \in \mathcal{A}^{n \times n}$ be a pairwise comparison matrix. Its \emph{opposite} is the pairwise comparison matrix $\mathbf{A}^- = \left[ 1 / a_{ij} \right] \in \mathcal{A}^{n \times n}$.
\end{definition}

Taking the opposite is equivalent to reversing all preferences of the decision-maker, or considering the transpose of the original pairwise comparison matrix.

\begin{axiom} \label{Axiom5}
\emph{Inversion} ($INV$):
Let $\mathbf{A} \in \mathcal{A}^{n \times n}$ be a pairwise comparison matrix.
Ranking method $g: \mathcal{A}^{n \times n} \to \mathfrak{R}^n$ is \emph{invertible} if $i \succeq^g_\mathbf{A} j \iff i \preceq^g_{\mathbf{A}^-} j$ for all $1 \leq i,j \leq n$.
\end{axiom}

Inversion means that reversing all preferences changes the ranking accordingly.

\section{Analysis of the Eigenvector Ranking Method} \label{Sec4}

\begin{lemma} \label{Lemma41}
A ranking method $g: \mathcal{A}^{n \times n} \to \mathfrak{R}^n$, satisfying $ANO$ and $AI$, meets $INV$.
\end{lemma}

\begin{proof}
Assume to the contrary that there exist alternatives $i$ and $j$ with a pairwise comparison matrix $\mathbf{A}$ such that $i \succeq^g_{\mathbf{A}} j$ and $i \succ^g_{\mathbf{A}^-} j$.
Consider the aggregated pairwise comparison matrix $\mathbf{A} \oplus \mathbf{A}^- = \mathbf{1}$. Anonymity implies $i \sim^g_{\mathbf{1}} j$, while aggregation invariance leads to $i \succ^g_{\mathbf{1}} j$, a contradiction.
\end{proof}

\begin{lemma} \label{Lemma42}
The Eigenvector Method satisfies anonymity.
\end{lemma}

\begin{proof}
A ranking method is guaranteed to be anonymous if it is derived from a weighting method satisfying comparison order invariance, a property introduced by \citet{Fichtner1984}. \citet[p.~344]{Fichtner1986} mentions that the Eigenvector Method fulfil comparison order invariance.
\end{proof}

\begin{lemma} \label{Lemma43}
The Eigenvector Method may violate inversion.
\end{lemma}

\begin{proof}
It is due to right-left asymmetry \citep{JohnsonBeineWang1979}, the difference of the rankings induced by the right and inverse left eigenvectors.
Using the example of \citet{DoddDoneganMcMaster1995}:
\[
\mathbf{A} = \left[
\begin{array}{K{1.5em} K{1.5em} K{1.5em} K{1.5em} K{1.5em}}
    1     & 1     & 3     & 9     & 9 \\
    1     & 1     & 5     & 8     & 5 \\
    1/3   & 1/5   & 1     & 9     & 5 \\
    1/9   & 1/8   & 1/9   & 1     & 1 \\
    1/9   & 1/5   & 1/5   & 1     & 1 \\
\end{array}
\right]:
\mathbf{w}^{EM}(\mathbf{A}) = \left[
\begin{array}{c}
    0.3657 \\
    0.3896 \\
    0.1672 \\
    0.0347 \\
    0.0429 \\
\end{array} 
\right] \text{; }
\mathbf{w}^{EM}(\mathbf{A}^-) = \left[
\begin{array}{c}
    0.0388 \\
    0.0432 \\
    0.1045 \\
    0.4580 \\
    0.3555 \\
\end{array} 
\right],
\]
therefore $1 \prec^{EM}_{\mathbf{A}} 2$ and $1 \prec^{EM}_{\mathbf{A}^-} 2$. Note that the principal eigenvalue is $\lambda(\mathbf{A}) \approx 5.348$ and the inconsistency ratio is $CR(\mathbf{A}) \approx 0.078$.

Furthermore, a counterexample with four alternatives and an inconsistency ratio below the 10\% threshold can be found, too (a $4 \times 4$ pairwise comparison matrix of \citet{JohnsonBeineWang1979} has an unacceptable level of inconsistency):
%\footnote{~In a pairwise comparison matrix of size $3 \times 3$, the left and right eigenvectors induce the same ranking as the Logarithmic Least Squares Method \citep{CrawfordWilliams1985}, which trivially satisfies inversion.}
\[
\mathbf{B} = \left[
\begin{array}{K{1.5em} K{1.5em} K{1.5em} K{1.5em}}
    1     & 1     & 1     & 9     \\
    1     & 1     & 2     & 5     \\
    1     &  1/2  & 1     & 9     \\
     1/9  &  1/5  &  1/9  & 1     \\
\end{array}
\right]:
\mathbf{w}^{EM}(\mathbf{B}) = \left[
\begin{array}{c}
    0.3242 \\
    0.3502 \\
    0.2821 \\
    0.0435 \\
\end{array} 
\right] \text{; }
\mathbf{w}^{EM}(\mathbf{B}^-) = \left[
\begin{array}{c}
    0.0886 \\
    0.0905 \\
    0.1104 \\
    0.7105 \\
\end{array} 
\right],
\]
hence $1 \prec^{EM}_{\mathbf{B}} 2$ and $1 \prec^{EM}_{\mathbf{B}^-} 2$. The principal eigenvalue is $\lambda(\mathbf{B}) \approx 4.158$, so the inconsistency ratio is $CR(\mathbf{B}) \approx 0.06$.
\end{proof}

\begin{proposition} \label{Prop41}
The Eigenvector Method may violate aggregation invariance.
\end{proposition}

\begin{proof}
It immediately follows from Lemmata~\ref{Lemma41}, \ref{Lemma42} and \ref{Lemma43}.
\end{proof}

\begin{lemma} \label{Lemma44}
The Eigenvector Method satisfies aggregation invariance if the number of alternatives is at most three.
\end{lemma}

\begin{proof}
The Eigenvector Method coincides with the Logarithmic Least Squares Method in the case of $n=3$ \citep{CrawfordWilliams1985}, and the latter obviously satisfies aggregation invariance.
\end{proof}

\begin{proposition} \label{Prop42}
Aggregation invariance and group-coherence for choice are equivalent properties in the case of a weighting method  satisfying invariance to row multiplication.
\end{proposition}

\begin{proof}
$AI \Rightarrow GCC$: if aggregation invariance holds, there can be no pair of alternatives for which a reversion in the order of preferences happens, and, in particular, there can be no reversion in the order of preferences between the alternative with the highest priority and another alternative. Therefore, group-coherence for choice holds.

$GCC \Rightarrow AI$: Suppose that a weighting method $f: \mathcal{A}^{n \times n} \to \mathcal{R}^{n}$ does not satisfy $AI$, namely, there exist some pairwise comparison matrices $\mathbf{A}^{(1)},\mathbf{A}^{(2)}, \dots , \mathbf{A}^{(k)} \in \mathcal{A}^{n \times n}$ such that $f_i \left( \mathbf{A}^{(\ell)} \right) \geq f_j \left( \mathbf{A}^{(\ell)} \right)$, but $f_i \left( {\mathbf{A}^{(1)} \oplus \mathbf{A}^{(2)} \oplus \dots \oplus \mathbf{A}^{(k)}} \right) < f_j \left( {\mathbf{A}^{(1)} \oplus \mathbf{A}^{(2)} \oplus \dots \oplus \mathbf{A}^{(k)}} \right)$, or $f_i \left( {\mathbf{A}^{(1)} \oplus \mathbf{A}^{(2)} \oplus \dots \oplus \mathbf{A}^{(k)}} \right) = f_j \left( {\mathbf{A}^{(1)} \oplus \mathbf{A}^{(2)} \oplus \dots \oplus \mathbf{A}^{(k)}} \right)$ and $f_i \left( \mathbf{A}^{(\ell)} \right) > f_j \left( \mathbf{A}^{(\ell)} \right)$ for at least one $1 \leq \ell \leq n$.

Consider the pairwise comparison matrix $\mathbf{A}^{(\ell)}$.
Choose a number $\alpha^{(\ell)} > \max \{ f_m \left( \mathbf{A}^{(\ell)} \right): 1 \leq m \leq n \} / f_i \left( \mathbf{A}^{(\ell)} \right)$.
Let $\hat{\mathbf{A}}^{(\ell)}$ be the pairwise comparison matrix obtained from $\mathbf{A}^{(\ell)}$ through two transformations of row multiplication on $i$ and $j$ by $\alpha^{(\ell)}$ (in an arbitrary order). Consequently, $f_i \left( \hat{\mathbf{A}}^{(\ell)} \right) / f_j \left( \hat{\mathbf{A}}^{(\ell)} \right) = f_i \left( \mathbf{A}^{(\ell)} \right) / f_j \left( \mathbf{A}^{(\ell)} \right) \geq 1$ and $f_i \left( \hat{\mathbf{A}}^{(\ell)} \right) \geq f_m \left( \hat{\mathbf{A}}^{(\ell)} \right)$ for all $1 \leq m \leq n$ due to $IRM$.

Note that ${\hat{\mathbf{A}}^{(1)} \oplus \hat{\mathbf{A}}^{(2)} \oplus \dots \oplus \hat{\mathbf{A}}^{(k)}}$ can also be obtained from ${\mathbf{A}^{(1)} \oplus \mathbf{A}^{(2)} \oplus \dots \oplus \mathbf{A}^{(k)}}$ through two transformations of row multiplication on $i$ and $j$ by $\sqrt[k]{\prod_{\ell = 1}^n \alpha^{(\ell)}}$. Thus invariance to row multiplication of this particular weighting method $f$ results in
\[
\frac{f_i \left( {\hat{\mathbf{A}}^{(1)} \oplus \hat{\mathbf{A}}^{(2)} \oplus \dots \oplus \hat{\mathbf{A}}^{(k)}} \right)}{f_j \left( {\hat{\mathbf{A}}^{(1)} \oplus \hat{\mathbf{A}}^{(2)} \oplus \dots \oplus \hat{\mathbf{A}}^{(k)}} \right)} = \frac{f_i \left( {\mathbf{A}^{(1)} \oplus \mathbf{A}^{(2)} \oplus \dots \oplus \mathbf{A}^{(k)}} \right)}{f_j \left( {\mathbf{A}^{(1)} \oplus \mathbf{A}^{(2)} \oplus \dots \oplus \mathbf{A}^{(k)}} \right)} < 1,
\]
which implies the violation of $GCC$.

As an illustration of the proof above, take pairwise comparison matrix $\mathbf{B}$ from the counterexample of Lemma~\ref{Lemma43}. Here $2 \succ^{EM}_{\mathbf{B}} 1 \succ^{EM}_{\mathbf{B}} 3 \succ^{EM}_{\mathbf{B}} 4$, but $4 \succ^{EM}_{\mathbf{B}^-} 3 \succ^{EM}_{\mathbf{B}^-} 2 \succ^{EM}_{\mathbf{B}^-} 1$. Actually, $\max \{ w_m^{EM} \left( \mathbf{B}^- \right): 1 \leq m \leq 4 \} / w_2^{EM} \left( \mathbf{B}^- \right) \approx 7.8478$, so $\alpha = 9$ is an appropriate choice for row multiplications on $1$ and $2$:
\[
\hat{\mathbf{B}^-} = \left[
\begin{array}{K{1.5em} K{1.5em} K{1.5em} K{1.5em}}
    1     & 1     & 9     & 1     \\
    1     & 1     & 9/2   & 9/5   \\
    1/9   &  2/9  & 1     & 1/9   \\
    1     &  5/9  & 9     & 1     \\
\end{array}
\right]: \quad
\mathbf{w}^{EM} \left( \hat{\mathbf{B}^-} \right) = \left[
\begin{array}{c}
    0.3278 \\
    0.3349 \\
    0.0454 \\
    0.2920 \\
\end{array} 
\right];
\]
\[
\mathbf{B} \oplus \hat{\mathbf{B}^-} = \left[
\begin{array}{K{1.5em} K{1.5em} K{1.5em} K{1.5em}}
    1     & 1     & 3     & 3     \\
    1     & 1     & 3     & 3     \\
     1/3  &  1/3  & 1     & 1     \\
     1/3  &  1/3  & 1     & 1     \\
\end{array}
\right]: \quad
\mathbf{w}^{EM} \left( \mathbf{B} \oplus \hat{\mathbf{B}^-} \right) = \left[
\begin{array}{c}
    3/8 \\
    3/8 \\
    1/8 \\
    1/8 \\
\end{array} 
\right].
\]
Hence $2 \succ^{EM}_{\hat{\mathbf{B}^-}} 1 \succ^{EM}_{\hat{\mathbf{B}^-}} 4 \succ^{EM}_{\hat{\mathbf{B}^-}} 3$, but $2 \sim^{EM}_{\mathbf{B} \oplus \hat{\mathbf{B}^-}} 1$.
\end{proof}

\begin{corollary} \label{Cor41}
The Eigenvector Method may violate group-coherence for choice. % because it may fail inversion.
\end{corollary}

\begin{proof}
It can be derived from Lemmata~\ref{Lemma41}, \ref{Lemma42} and \ref{Lemma43} (violation of aggregation invariance is a consequence of failing inversion) together with Lemma~\ref{Lemma21} and Proposition~\ref{Prop42} (violation of aggregation invariance is equivalent to failing group coherence for choice).
\end{proof}

The counterexample of \citet{PerezMokotoff2016} contains five alternatives. We have shown that the Eigenvector Method may violate $GCC$ even if $n=4$, but this example cannot be simplified with respect to the number of alternatives (see Lemma~\ref{Lemma44}).

\section{Conclusions} \label{Sec5}

We have discussed a powerful argument against the use of Eigenvector Method in group decision-making: it may happen that the relative priorities derived from the aggregated pairwise comparison matrix do not reflect the common individual preferences even if there are only four alternatives and all decision-makers provide a pairwise comparison matrix with an acceptable level of inconsistency. This negative result is proved to be a consequence of the right-left asymmetry of eigenvectors, the possibly different ranking of alternatives from a given pairwise comparison matrix and its tranpose.

Investigation of other ranking (weighting) methods with respect to these axioms remains a topic of future research.

\section*{Acknowledgements}
\addcontentsline{toc}{section}{Acknowledgements}
\noindent
I am grateful to S\'andor Boz\'oki for useful advices. \\
Three anonymous reviewers provided valuable comments and suggestions on earlier drafts. \\
The research was supported by OTKA grant K 111797 and by the MTA Premium Post Doctorate Research Program. \\
This research was partially supported by Pallas Athene Domus Scientiae Foundation. The views expressed are those of the author's and do not necessarily reflect the official opinion of Pallas Athene Domus Scientiae Foundation. %\\
%This research was partially supported by Pallas Athene Domus Scientiae Foundation. The views expressed are those of the author's and do not necessarily reflect the official opinion of Pallas Athene Domus Scientiae Foundation.

%\bibliographystyle{apalike}
%\bibliographystyle{unsrtnat}
%\bibliography{All_references}
%\addcontentsline{toc}{section}{References}

\end{document}